\documentclass[12pt]{article}
\usepackage{amsfonts}
\usepackage{amssymb,amsmath,amsthm,xy}
\usepackage{MnSymbol}
\xyoption{all}
\usepackage[left=2.8cm,right=2.8cm,top=2cm,bottom=2cm,includehead]{geometry}
\usepackage{hyperref}
\usepackage{enumerate}
\usepackage{appendix}
\usepackage{fancyhdr}

\hypersetup{colorlinks = true, linkcolor = blue, urlcolor = blue, citecolor=blue}

\newtheoremstyle{myplain}
{}
{}
{\itshape}
{}
{\bfseries}
{:}
{.5em}
{}

\theoremstyle{myplain}
\newtheorem{lem}{Lemma}
\newtheorem{defn}{Definition}
\newtheorem{thm}{Theorem}

\newtheorem*{notn}{Notation}

\newtheorem*{clm}{Claim}

\renewcommand{\L}{\mathcal{L}}
\newcommand{\M}{\mathcal{M}}
\newcommand{\N}{\mathcal{N}}

\newcommand{\R}{\mathcal{R}}
\renewcommand{\S}{\mathcal{S}}
\newcommand{\U}{\mathcal{U}}

\newcommand{\KK}{\mathbb{K}}
\newcommand{\NN}{\mathbb{N}}

\newcommand{\RR}{\mathbb{R}}
\newcommand{\ZZ}{\mathbb{Z}}

\renewcommand{\int}{\textnormal{int}} 
\newcommand{\bd}{\textnormal{bd}} 
\newcommand{\Con}{\textnormal{Con}} 
\newcommand{\omin}{\textnormal{o-min}} 
\newcommand{\T}{\textnormal{T}} 
\newcommand{\Th}{\textnormal{Th}} 

\title{\normalsize\textbf{THE NON-AXIOMATIZABILITY OF O-MINIMALITY}}
\author{\normalsize ALEX RENNET}
\date{}

\begin{document}

\pagestyle{fancy}
\fancyhead{}
\fancyfoot{}
\fancyhead[CO,CE]{Non-Axiomatizability \hfill \hspace{-1.25in}\textbf{Alex Rennet} \hfill  \thepage \ of \pageref*{Last Page}}

\maketitle
\vspace{-0.25in}
\begin{abstract}
Fix a language $\L$ extending the language of real closed fields by at least one new predicate or function symbol.  Call an $\L$-structure $\R$ \textit{pseudo-o-minimal} if it is (elementarily equivalent to) an ultraproduct of o-minimal structures.  We show that for any recursive list of $\L$-sentences $\Lambda$, there is a real closed field $\R$ satisfying $\Lambda$ which is not pseudo-o-minimal.  In particular, there are locally o-minimal, definably complete real closed fields which are not pseudo-o-minimal.  This answers negatively a question raised by Schoutens in \cite{Schoutens}, and shows that the theory $\T^{\omin}$ consisting of those $\L$-sentences true in \textit{all} o-minimal $\L$-structures, also called the \textit{theory of o-minimality (for $\L$)}, is not recursively axiomatizable.  
\end{abstract}


\section{Introduction}

In \cite{Ax}, Ax proved that the \textit{theory of finite fields}, consisting of all those sentences in the language of fields which are true in all finite fields, is recursively axiomatizable.  He showed first that the infinite fields which are (elementarily equivalent to) ultraproducts of finite fields, called \textit{pseudofinite fields}, are precisely those fields which are perfect, pseudoalgebraically closed and have an algebraic extension of each degree.  And second, he showed that these properties are all first-order definable by recursive axiom schemas. Another positive result along the same lines comes from \textit{pseudofinite linear orderings}: the theory of finite orderings (i.e. those $\L_{<}$-sentences true in all finite orderings) is axiomatized by the statements that the ordering is discrete and has a first and last element (see \cite{Va01}).  A linear ordering is elementarily equivalent to an ultraproduct of finite linear orderings if and only if it is finite or has order type $\omega + L \cdot \ZZ + \omega^{*}$ for some linear order $L$.  

Both of these cases fall into the following general framework: fix a language $\L$, let $\KK$ be a class of $\L$-structures and let $\Th(\KK)$ be the theory consisting of those $\L$-sentences which are true in all models $\M \in \KK$.  Then the class of models of $\Th(\KK)$ is exactly the class $\overline{\KK}$ obtained by closing $\KK$ under isomorphism, elementary equivalence and ultraproducts (this follows from \cite[Corollary 8.5.13]{Hodges}). Now, the axiomatizability results above say that with $\KK$ the class of finite fields, or the class of finite linear orderings, the theory $\Th(\KK)$ is recursively axiomatizable.  

\textit{O-minimality} is a condition on ordered structures which states that every definable subset of the line is a finite union of points and intervals (see \cite{VdD} for a thorough introduction to o-minimality). This property is \textit{not} first-order expressible, by which we mean either of two equivalent defintions. The first is that there is no axiom schema $(\varphi_{i})_{i\in I}$ of $\L$-sentences such that an $\L$-structure $\R$ is o-minimal if and only if $\R\models \varphi_{i}$ for all $i \in I$.  By Los' Theorem, this is equivalent to saying that an ultraproduct of o-minimal structures need not be o-minimal.  In fact, not only might such a structure not be o-minimal, it might not be NIP (see e.g. \hspace{-0.05in}\cite[Example 6.19]{FornasieroTame}).  We should note that a property being first-order in either of these senses is a condition that is separate from the condition of it being preserved under elementary equivalence, even though the two can easily be confused (the word `elementary' being so common does not make it any easier).  In some cases, these two conditions either obtain or fail to obtain together, but in the case of o-minimality, they are not the same: it was shown early on in the study of o-minimality (see \cite{KPS86}) that for any two elementarily equivalent structures, either both or neither is o-minimal.  

The moral of all this is that since o-minimality is not first-order in our sense(s), in order to axiomatize o-minimality in a way analogous to the situation with finite fields and finite linear orderings, we would need to collect a generating set of axioms for all the first-order consequences of o-minimality.  And again since o-minimality is not first-order, we can regard any such consequence as a (first-order) \textit{weakening} of o-minimality.

Many weakenings of o-minimality, such as \textit{weak o-minimality} (\cite{Macpherson-Marker-Steinhorn}), \textit{quasi-o-minimality} (\cite{BPW00}), \textit{d-minimality} (\cite{Mi05}), \textit{o-minimal open core} (\cite{DMS09}) etc., have been studied in the literature (for more, see \cite{VdD98,Toff-Voz,Mi01,FornasieroTame}).  However, since we are only interested in first-order weakenings, the two that are of particular interest to us here are the following two first-order conditions: \textit{local o-minimality (LOM)} (that for every definable subset of the line, and every point, there is a neighbourhood of that point where the definable set is a finite union of points and intervals)\footnote{If we strenthen this by allowing the point to possibly be $\pm\infty$, we get a property that Schoutens calls \textit{type completeness} in \cite{Schoutens}.  Local o-minimality, though prima facie weaker, is only weaker if the structure in question does not have a \textit{multiplicative} group structure, which will not be the case for us.} 
and \textit{definable completeness (DC)} (that every bounded definable subset of the line has a supremum).  Ordered fields satisfying both of these properties are real closed and have particularly nice definable sets: every definable $A \subseteq R^{1}$ has a discrete boundary which first, has no accumulation points in the topology on $R$, and second, is either finite, or has order-type $\omega + L\cdot\ZZ + \omega^{*}$ for some linear order $L$.  Since LOM and DC are expressible by first-order axiom schemas (this is an easy lemma), they are true in \textit{every} o-minimal structure. Thus, they are a starting point to look for a recursive axiomatization of o-minimality. In \cite{Schoutens}, Schoutens hypothesizes that LOM and DC together (from now on, just LOM+DC) possibly with a first-order variant of the pigeonhole principle for discrete definable sets, might be enough to axiomatize $\T^{\omin}$.  Separately, in \cite{FornasieroTame}, Fornasiero investigated LOM+DC fields and conjectured that the aforementioned version of the pigeonhole principle follows from LOM+DC.  We will show that LOM+DC, even (potentially) strengthened by this version of the pigeonhole principle, or in fact strengthened by \textit{any} recursive list of axioms, does not axiomatize $\T^{\omin}$:

\begin{thm}
Let $\L$ be a language extending the language of real closed fields with at least one new function or predicate symbol, and let $\Lambda$ be a recursive list of $\L$-sentences.  Then there is an $\L$-structure $\R_{\Lambda}^{\L}$, which satisfies LOM+DC+RCF+$\Lambda$ but is not elementarily equivalent to an ultraproduct of o-minimal $\L$-structures. 
\end{thm}

The rest of the paper has two parts: in the first, we set-up some terminology and review some necessary definitions and a lemma; in the second, we prove Theorem 1.

I would like to thank Ehud Hrushovski for discussions at the 2010 AMS Mathematics Research Community on the Model Theory of Fields during which he told me that G\"{o}del's Second Incompleteness Theorem implied that pseudofiniteness in general was not recursively axiomatizable.  I would also like to thank Jana Marikova for helpful comments on an earlier draft of this paper.

This paper is a condensed version of a chapter of my dissertation.  

\section{Preliminaries}

Let $\L$ extend the language of real closed fields, and let $\R$ be an $\L$-structure which is an expansion of a real closed field. For a subset $X\subset \R$, we define $X^{\leq r}:=\{x \in X \ | \ x\leq r\}$, and similarly for $X^{\geq r}$, $X^{<r}$ and $X^{>r}$.  We write $\bd(X)$ for the topological boundary of $X$.

The following lemma will be used later, and is equivalent to part of \cite[Theorem 3.14]{Schoutens}:

\begin{lem}
\label{lemDcbLOMDC}
If for every definable $X \subseteq \R$, we have that $\bd(X)$ is discrete, closed and bounded and (if non-empty) has a least and greatest element, then $\R$ satisfies LOM+DC.
\end{lem}

\begin{proof}
Suppose that $X\subseteq \R$ is definable, and bounded above.  In order for DC to hold, we need only show that $X$ has a supremum in $\R$.  But, by the definition of $\bd(X)$, its greatest element is certainly the supremum of $X$. 

Before we show that LOM follows, note the following claim:

\begin{clm} For every non-empty definable $X \subsetneq \R$, every element of $\bd(X)$ has a unique successor and predecessor (except the greatest and least elements, respectively.)  
\end{clm}

Let $x \in \bd(X)$ be a non-least element, if such an element exists, and consider the set $\bd(X)^{<x}$: this set is definable, and since it equals its own boundary, it is the boundary of a definable set, and so has a greatest element by the assumption of the lemma, and the choice of $x$.  This element is the unique predeccessor of $x$ in $\bd(X)$.  Similarly, we can determine the unique successor of $x$ (if it is not the greatest element of $\bd(X)$).

Returning to the proof, we continue by showing LOM.  Let $X\subsetneq \R$ be non-empty and definable, let $r \in \R$, and consider $\bd(X^{\leq r})$: if this set has a greatest element less than $r$, define $a$ to be that element.  If the greatest element is $r$, define $a$ to be the predeccessor of $r$ in $\bd(X^{\leq r})$, if it exists.  If no such element exists, then $\bd(X^{\leq r}) \subset \{r\}$, and thus $(-\infty,r) \subset X$ or $X^{<r}=\emptyset$; and we can then define $a$ to be $r -1$.  Similarly, if $\bd(X^{\geq r}) \subset \{r\}$, define $b$ to be $r+1$, and otherwise define $b$ to be the least element of $\bd(X^{\geq r})^{>r}$.  By choosing $a$ and $b$ in this way, we have ensured that $X \cap (a,b)$ has at most two components, and since $r\in\R$ was arbitrary, LOM follows.
\end{proof}

\noindent \textbf{Terminology:} Though the word \textit{o-minimalistic} is used in \cite{Schoutens} to describe those structures elementarily equivalent to ultraproducts of o-minimal structures, the terminology \textit{pseudo-o-minimal} is more in line with examples like pseudofinite fields and pseudofinite orderings, so we shall use the latter from now on.\\

\noindent \textbf{Note:} In the language of real closed fields $\L_{RCF}$, any ordered field considered as an $\L_{RCF}$-structure with the usual interpretations of the symbols is o-minimal if and only if it is real closed.  Since being a real closed field is first-order, ultraproducts of real closed fields are real closed fields, and hence o-minimal.  Thus, in the bare field language, pseudo-o-minimal structures are o-minimal.  So, from now on we will assume that $\L$ extends $\L_{RCF}(N)$, the language of real closed fields extended by a new unary predicate $N$. 

\begin{notn} 
If $\M$ is an ultraproduct $\prod_{i \in I}\M_{i}/\U$, and $\U$ is an ultrafilter on $I$, then if $$\{i \in I \ | \ \M_{i} \textnormal{ \textit{has property} P}\} \in \U \textnormal{ \textit{for some property} P,}$$ we will say that ``$\U$-most'' index models $\M_{i}$ have property \textnormal{P}.
\end{notn}

Let $\KK^{\omin}$ be the class of all o-minimal $\L$-structures, and let $\T^{\omin}$ be the set of $\L$-sentences $\varphi$ such that $\M \models \varphi$ for all $\M \in \KK^{\omin}$.  
We call $\T^{\omin}$ the \textit{theory of o-minimality for }$\L$.  

Finally, for every one-variable $\L$-formula $\varphi(z)$, we fix a new variable $x$ not occuring in $\varphi(z)$ and make the following definition: 

\begin{defn}
\label{defnVarphiLeqX}
$\varphi^{\leq x}(z)$ is the result of taking $\varphi$, and whenever `$N(t)$' appears for some $\L$-term $t$, replacing it with `$N(t) \land t \leq x$'. 
\end{defn}

\section{Proof of Theorem 1}
\hspace{0.35in}Fix from now on a recursive list $\Lambda$ of $\L$-sentences which will be a purported axiomatization (with LOM+DC+RCF) of $\T^{\omin}$.

Let PA be the \textit{relational} theory of Peano Arithmetic.  That is, we consider PA in a language with relation symbols (only) for addition, multiplication and ordering.  Note that this is essentially equivalent to the usual theory PA in a language with function symbols for addition, multiplication, and successor, in that a model of relational PA, can be definitionally expanded to a model of ordinary PA, and vice versa. The main difference is that a substructure of a model of relational PA can be finite. 

From now on, extend $\L$ by two ternary predicate symbols $\alpha$ and $\mu$, and let T be the $\L$-theory consisting of the following informally stated, but nonetheless first-order axiom schemas (where $(R,+,\times,<,0,1, N, \alpha,\mu)$ is a model of the axioms):

\begin{enumerate}[(I)]
\item $\overline{R}:=(R,  +, \times, <, 0, 1)\models RCF$
\item $(N, \alpha, \mu, <\upharpoonright_{N},0,1) \models PA$
\item $\alpha = +\hspace{-0.04in}\upharpoonright_{N}$ and $\mu = \times\hspace{-0.04in}\upharpoonright_{N}$
\item $\mathbf{(LOM+DC)}_{\varphi(z) \in \L}  : (\overline{R}, N, \alpha, \mu) \models$ $\forall x, \ \bd(\varphi^{\leq x}(\R))$ is discrete, closed and bounded and (if non-empty) has a least and greatest element
\item $\mathbf{(\Lambda)}_{\psi \in \Lambda}  : (\overline{R},N,\alpha,\mu) \models \forall x \ \psi^{\leq x}$

\end{enumerate}

Unpacking this a bit, (I) contains all of the standard axioms for real closed fields, (II) states that all of the axioms for PA hold where all the quantifiers are relativized to the predicate $N$ (i.e. instead of ``$\forall x \ (...)$'', we have ``$\forall x \ (N(x) \rightarrow ...)$'' etc.), and where ``$+$'' is replaced by $\alpha$, and ``$\times$'' is replaced by $\mu$, (III) asserts that $\alpha$ and $\mu$ are strictly subsets of $N^{3}$, (IV) ensures that when the model $N$ of PA is restricted to any initial segment, the set defined by $\varphi$ in $R$ with this initial segment of $N$ has a discrete, closed and bounded boundary, and (V) forces every axiom in $\Lambda$ to hold in $R$ with $N$ again restricted to any fixed initial segment. \\

We will show that not only does T have a model, but we will then show that there is a model of T with a reduct that satisfies LOM+DC+RCF+$\Lambda$ but which could not possibly be an ultraproduct of o-minimal structures.

In order to accomplish this, we first note that T is consistent.  To see this, observe that the real field, $\overline{\RR} = (\RR, +, \times, <, 0, 1)$ with an added predicate for $\NN$ is a model of T: for (I), (II) and (III), this is clear; now for (IV) and (V) we can see that for any $r\in \RR$, $\NN^{\leq r}$ is a \textit{finite} initial segment of $\NN$, so this subset was definable inside $\overline{\RR}$ already; indeed, the partial ternary subsets corresponding to restricted multiplication and addition (i.e. the interpretations of $\mu$ and $\alpha$ respectively) on this initial segment could already be defined in $\overline{\RR}$ by the graphs of the restricted multiplication and addition functions on this initial segment.  Thus, $(\overline{\RR},\NN^{\leq r})$ is just a definitional expansion of $\overline{\RR}$.  Now, since $\overline{\RR}$ is o-minimal, $(\overline{\RR},\NN^{\leq r})$ thus satisfies all the first-order consequences of o-minimality; in particular, it satisfies LOM+DC+RCF+$\Lambda$.  Thus, $(\overline{\RR},\NN)\models$ T, so the theory T is consistent.

However, since every model of T interprets a model of PA (in fact, it defines one), G\"{o}del's Second Incompleteness Theorem applies, allowing us to conclude that $\T + \neg \Con(\T)$ is also consistent.  That is, there is a model $(\R,\N)$ of $\T + \neg \Con(\T)$.  In particular, $\N$ has a code, say the element $\alpha\in\N$, for a proof of $\neg \Con(\T)$.

But then, letting $\alpha<x \in \R$ be sufficiently large, we have that the code for the proof of $\neg \Con(\T)$ and the codes for any symbols occuring in its proof are contained in $\N^{\leq x}$.  Since $(\R,\N)$ satisfies (IV), we have that in $(\R,\N^{\leq x})$, the boundary of every definable subset of the line is discrete closed and bounded, and (if non-empty) has a least and greatest element; thus by Lemma \ref{lemDcbLOMDC}, $(\R,\N^{\leq x})$ satisfies LOM+DC.  And since $(\R,\N)$ satisfies (V), $(\R,\N^{\leq x})$ satisfies $\Lambda$.  Finally, since $\N^{\leq x}$ is an initial segment of a model $\N$ of PA with bounded portions of addition and multiplication, it is a $\Delta_{0}$-elementary substructure of $\N$ (by \cite[Theorem 2.7]{Kaye}).  Since $\alpha < x$ and $x$ is sufficiently large in $\N$, and since $\alpha$ being a code for a proof of $0=1$ in T is a $\Delta_{0}$-property of $\alpha \in \N^{\leq x}$, we have that $\N^{\leq x} \models \neg \Con(\T)$.

But $(\R,\N^{\leq x})$ could not possibly be pseudo-o-minimal.  Suppose for contradiction that it was elementarily equivalent to an ultraproduct of o-minimal structures: $$(\R,\N^{\leq x}) \equiv (\S,\M) = \prod_{i \in I}(\S_{i},\M_{i})/\U$$ with $\U$ a non-principal ultrafilter on $I$, and $\U$-most $(\S_{i}, \M_{i})$ o-minimal.
Then we would have $(\S,\M) \models \neg \Con(\T)$ by elementary equivalence. 
But since $\M$ is discrete, $\U$-most of the sets $\M_{i}$ must be discrete by Los' Theorem.  And since $\U$-most index models $(\S_{i},\M_{i})$ are o-minimal, $\U$-most $\M_{i}$ must then be finite, being discrete and definable.  
Since $\N^{\leq x}$ is an initial segment of a model of PA, so is $\M$.  Since $\U$-most of the $\M_{i}$ are finite, $\U$-most of them are finite initial segments; but being a finite initial segment is the same as being isomorphic to the structure $N_{n}$ consisting, for some $n\in\NN$, of the first $n$ elements of $\NN$, together with the graphs of addition, multiplication, and ordering restricted to this set.  That is, $\U$-most $\M_{i}$ are isomorphic, for some $n_{i}$ to the structure $N_{n_{i}}=(\{0,1,...,n_{i}\}, \alpha\upharpoonright_{N_{n_{i}}}, \mu\upharpoonright_{N_{n_{i}}}, <\upharpoonright_{N_{n_{i}}})$.

Finally, $\M\models\neg\Con(\T)$, so there is $\alpha\in\M$ such that $\alpha$ is a code for a proof of $0=1$ in T. But then for an index $i$ such that $(\S_{i},\M_{i})$ is o-minimal, and $\M_{i}$ is isomorphic to some $N_{n_{i}}$ as above, then the $i$-th coordinate of $\alpha$, i.e. the element $\alpha_{i}\in\M_{i}$, must be a code for a proof of $0=1$ in T as well.
But this implies that there is a \textit{standard} code for the proof of $\neg \Con(\T)$.  From the existence of a standard code for a proof we could recover an actual proof of $\neg\Con(\T)$.  Hence, T would actually be inconsistent, a contradiction.

\vspace{-0.2in}\begin{flushright}$\square$\end{flushright}

\noindent\texttt{Email: \href{rennetad@math.berkeley.edu}{rennetad@math.berkeley.edu} }\\
\texttt{Web: \href{http://math.berkeley.edu/~rennetad}{http://math.berkeley.edu/$\sim$rennetad}}
\label{Last Page}
\end{document}